\documentclass{article}

\usepackage{arxiv}

\usepackage[utf8]{inputenc} 
\usepackage[T1]{fontenc}    
\usepackage{hyperref}       
\usepackage{url}            
\usepackage{booktabs}       
\usepackage{amsfonts}       
\usepackage{nicefrac}       
\usepackage{microtype}      
\usepackage{graphicx}
\usepackage{doi}
\usepackage[
color=orange!80, 
bordercolor=black,
textwidth=3cm,
textsize=small,
colorinlistoftodos]
{todonotes}
\usepackage{comment}

\usepackage[
backend=biber,
style=alphabetic,
maxbibnames=99
]{biblatex}
\usepackage{amsthm}
\usepackage{amsmath}
\usepackage{amsfonts}
\usepackage{ amssymb }
\usepackage{xfrac}
\usepackage{quiver}
\usepackage{float}
\usepackage{array}
\usepackage{scalerel}

\newcommand{\cat}[1]{\ensuremath{\mathbf{#1}}}

\newcommand{\Z}{\mathbb Z}
\newcommand{\R}{\mathbb R}
\newcommand{\N}{\mathbb N}

\newcommand{\Q}{\mathbb Q}
\newcommand{\T}{\mathbb T}
\newcommand{\F}{\mathbb F}

\newcommand{\Span}{{\operatorname{Span}}}

\newcommand{\trop}{{\operatorname{trop}}}

\newcommand{\gen}[1]{\left\langle#1\right\rangle}

\theoremstyle{definition}
\newtheorem{definition}{Definition}[]

\newcounter{remark}
\newenvironment{remark}{\refstepcounter{remark}\par
   \noindent \textbf{Remark~\theremark.} \rmfamily}{\medskip}

\theoremstyle{plain}
\newtheorem{theorem}{Theorem}[]
\newtheorem{corollary}[theorem]{Corollary}
\newtheorem{lemma}[theorem]{Lemma}
\newtheorem{proposition}[theorem]{Proposition}

\newenvironment{notation}{\par\noindent \textbf{Notation:} \rmfamily}{\medskip}

\newcounter{example}
\newenvironment{example}[1][]{\refstepcounter{example}\par
   \noindent \textbf{Example~\theexample. #1} \rmfamily}{\medskip}

\makeatletter
\def\renewtheorem#1{%
  \expandafter\let\csname#1\endcsname\relax
  \expandafter\let\csname c@#1\endcsname\relax
  \gdef\renewtheorem@envname{#1}
  \renewtheorem@secpar
}
\def\renewtheorem@secpar{\@ifnextchar[{\renewtheorem@numberedlike}{\renewtheorem@nonumberedlike}}
\def\renewtheorem@numberedlike[#1]#2{\newtheorem{\renewtheorem@envname}[#1]{#2}}
\def\renewtheorem@nonumberedlike#1{  
\def\renewtheorem@caption{#1}
\edef\renewtheorem@nowithin{\noexpand\newtheorem{\renewtheorem@envname}{\renewtheorem@caption}}
\renewtheorem@thirdpar
}
\def\renewtheorem@thirdpar{\@ifnextchar[{\renewtheorem@within}{\renewtheorem@nowithin}}
\def\renewtheorem@within[#1]{\renewtheorem@nowithin[#1]}
\makeatother

\makeatletter
\@ifpackageloaded{fdsymbol}\@tempswafalse\@tempswatrue
\if@tempswa
  \newcommand{\fdsy@scale}{1.0}
  \newcommand\fdsy@mweight@normal{Book}
  \newcommand\fdsy@mweight@small{Book}
  \newcommand\fdsy@bweight@normal{Medium}
  \newcommand\fdsy@bweight@small{Medium}
  \DeclareFontFamily{U}{FdSymbolA}{}
  \DeclareSymbolFont{fdsymbols}{U}{FdSymbolA}{m}{n}%
  \SetSymbolFont{symbols}{bold}{U}{FdSymbolA}{b}{n}%
  \DeclareFontShape{U}{FdSymbolA}{m}{n}{
      <-7.1> s * [\fdsy@scale] FdSymbolA-\fdsy@mweight@small
      <7.1-> s * [\fdsy@scale] FdSymbolA-\fdsy@mweight@normal
  }{}
  \DeclareFontShape{U}{FdSymbolA}{b}{n}{
      <-7.1> s * [\fdsy@scale] FdSymbolA-\fdsy@bweight@small
      <7.1-> s * [\fdsy@scale] FdSymbolA-\fdsy@bweight@normal
  }{}
  \DeclareMathSymbol{\aleph}{\mathord}{fdsymbols}{"C7}
  \DeclareMathSymbol{\beth}{\mathord}{fdsymbols}{"C8}
  \DeclareMathSymbol{\gimel}{\mathord}{fdsymbols}{"C9}
  \DeclareMathSymbol{\daleth}{\mathord}{fdsymbols}{"CA}
\fi
\makeatother

\title{Applications of Generalized Universal Valuations}

\author{William Bernardoni\\
	Case Western Reserve University\\
	Cleveland, OH 44118 \\
	\texttt{wrb37@case.edu}
}

\renewcommand{\headeright}{}
\renewcommand{\undertitle}{}

\renewcommand{\shorttitle}{Generalized Universal Valuations}

\hypersetup{
pdftitle={Applications of Generalized Universal Valuations},
pdfsubject={Generalized Universal Valuations},
pdfauthor={William Bernardoni},
pdfkeywords={Valuations (16W60), Non-Archimedean Valued Fields (12J25), Semirings (16Y60), Tropical Geometry (14T10)},
}

\begin{document}
\maketitle

\begin{abstract}
	We introduce a generalization of the universal valuation semiring defined by Jeffrey and Noah Giansiracusa. We then explicitly characterize the additive structure of this semiring and show that, when applied to $\Q$, this characterization gives the Non-Archimedean case of Ostrowski's theorem. We conclude with examples of non-commutative valuations and their applications, such as the detection of the existence of representations of rings in ultrametric vector spaces.
\end{abstract}


\section{Introduction}
Problems in Non-Archimedean geometry can be translated into problems in tropical geometry, the geometry of the commutative, totally ordered, tropical semiring. This translation, via the correspondence between Non-Archimedean absolute values and valuations into the tropical semiring, has introduced powerful tools into real algebraic geometry, such as the method of patchworking created by Oleg Viro to classify isotopy classes of real algebraic curves of degree 7 \cite{viro2006patchworking, Mikhalkin}.

In their paper on a scheme theoretic version of tropicalization \cite{2016}, the Giansiracusa brothers introduced a construction of a universal valuation semiring which allows for the study of valuations which are not totally ordered but are commutative. In this paper we will show that the commutativity assumption can be removed, and a family of non-commutative, non-totally ordered valuations can be studied via the maps out of a single universal valuation semiring.

Non-commutative, non-totally ordered idempotent semirings appear naturally in many applications such as in solving the algebraic path problem, or a whole slew of other important computational tasks which may be read about in \cite{Baras}. A recent application that the author studies is the behavior of satellite networks. Many problems in constructing a solar system internet can be reduced to algebro-geometric problems over the non-commutative, non-totally ordered idempotent semiring found in section 2.5 of \cite{CurrySpaceNetworking}.

As the correspondence between valuations into the tropical semiring and Non-Archimedean absolute values introduced tools to study tropical geometry, we hope that by classifying valuations into non-commutative, non-totally ordered idempotent semirings we may introduce tools that may be applied to these algebro-geometric problems.

\subsection*{Outline}
\begin{itemize}
    \item We begin in section \ref{sec:preliminaries} by introduce our objects of study, idempotent semirings (definitions \ref{def:semiring} and \ref{def:idempotent}) and generalized semiring valuations (definition \ref{def:valuation})
    \item We then give a description of a universal valuation semiring (definition \ref{UnivVal}) over not-necessarily-commutative rings in section \ref{sec:UnivSemi} as a slight generalization of the construction in \cite{2016}. We then give an explicit example of this construction over a non-commutative ring of order 8 in example \ref{ex:SmallestNC}.
    \item As the main theorem of this paper, theorem \ref{thm:valladdstruct}, in section \ref{sec:structure} we give an explicit characterization of the additive structure of this universal valuation semiring. We then show in section \ref{sec:NonArchOstrow} how applying this characterization to the rational numbers gives the Non-Archimedean case of Ostrowski's theorem as corollary \ref{thm:Ostrow}
    \item We then will motivate our choice to generalize the construction of the universal valuation semiring by looking at examples of non-commutative valuations in section \ref{sec:NoncommVal}, and show how the induced tropicalization of a non-commutative expression can give us tools to study the roots of those expressions in section \ref{sec:tropicalizations}, and see in section \ref{sec:representations} how a further generalized version of our valuations is related to the representations of a ring into an ultrametric vector space.
\end{itemize}

A reader familiar with semirings may skip section \ref{sec:semiring}, noting only that we use a sign convention for idempotent semirings in line with the min-plus tropical semiring:
\[a \le b \iff a + b = a\]


\section{Algebraic Preliminaries}
\label{sec:preliminaries}

\begin{notation}
    \begin{itemize}
        \item In this paper we assume semirings, rings, and monoids are unital.
        \item We denote the powerset of a set $X$ as $2^X$.
        \item In equations that involve multiple algebraic objects, we denote by $+_S$ and $*_S$ to be the addition and multiplication respectively in the object $S$. In particular when we refer to standard arithmetic we will note $+_\R$ and $*_\R$ to refer to the standard addition and multiplication over the real numbers.
        \item In general we will use $R$ to designate an arbitrary ring, $S$ an arbitrary semiring, and $\Gamma$ an arbitrary idempotent semiring.
    \end{itemize}
\end{notation}

\subsection{Semirings}\label{sec:semiring}
\begin{definition}\label{def:semiring}
A \textbf{semiring} $(S, +, *, 0_S, 1_S)$ is a tuple where $S$ is a set, $+$ and $*$ are binary operations on $S$, and $0_S$ and $1_S$ are elements of $S$ such that:
\begin{enumerate}
    \item $(S,+)$ is a commutative monoid with identity element $0_S$.
    \begin{itemize}
        \item $+$ is associative: $(a + b) + c = a + (b + c)$
        \item $0_S$ is an additive identity: $a + 0_S = a = 0_S + a$
        \item $+$ is commutative: $a + b = b + a$
    \end{itemize}
    \item $(S,*)$ is a monoid with identity element $1_S$.
    \begin{itemize}
        \item $*$ is associative: $a*(b*c) = (a*b)*c$
        \item $1_S$ is a multiplicative identity identity: $1_S * a = a = a * 1_S$
    \end{itemize}
    \item Multiplication distributes over addition.
    \begin{itemize}
        \item $a*(b + c) = a*b + a * c$
        \item $(a + b) * c = a*c + b * c$
    \end{itemize}
    \item Multiplication by $0_S$ annihilates $R$.
    \begin{itemize}
        \item $0_S * a = 0_S = a * 0_S$
    \end{itemize}
\end{enumerate}
\end{definition}

\begin{remark}
    Some authors refer to semirings as \textbf{Rigs}, as they are ``Rings without the \textbf{n}egatives." As we will be referring often to both rings and semirings, we will use the full term ``semiring" to avoid any confusion.
\end{remark}

\begin{definition}\label{def:idempotent}
A semiring $S$ is called \textbf{idempotent} if for all $a \in S$: $a + a = a$.
\end{definition}

\begin{proposition}\label{thm:idemp1}
    A semiring $S$ is idempotent if and only if $1_S + 1_S = 1_S$.
\end{proposition}
\begin{proof}
    Let $1_S + 1_S = 1_S$, then we note for any $x \in S$
    \[x = (1_S)x = (1_S + 1_S)x = x + x\]
    therefore $R$ is idempotent.

    Let $1_S + 1_S \ne 1_S$, then $S$ is not idempotent.
\end{proof}

\begin{definition}
    We refer to the category of idempotent semirings as $\cat{ISR}$.

    An object in $\cat{ISR}$ is an idempotent semiring, and a morphism $\phi : S \to {S'}$ is a unital semiring homomorphism. That is for all $a, b \in S$:
    \begin{itemize}
        \item $\phi(a +_S b) = \phi(a) +_{S'} \phi(b)$
        \item $\phi(a *_S b) = \phi(a) *_{S'} \phi(b)$
        \item $\phi(0_S) = 0_{S'}$
        \item $\phi(1_S) = 1_{S'}$
    \end{itemize}
\end{definition}

\begin{definition}
A semiring $S$ is called \textbf{commutative} if for all $a, b \in S$: $ab = ba$.
\end{definition}

\begin{example}
    The most famous example of a commutative idempotent semiring is the \textbf{tropical semiring} 
    \[\T = (\R \cup \{\infty\}, \min, +_\R, \infty, 0)\]
    The tropical numbers are the extended real numbers with addition being the minimum of two numbers and multiplication being standard addition, i.e.:
    \[a +_\T b = \min(a,b)\]
    \[a *_\T b = a +_\R b\]
    The tropical semiring has been studied in much detail, and some good introductions to the geometric and algebraic structure of $\T$ as well as its applications are the texts \cite{MaclaganSturmfels, Mikhalkin, MaxPlus}.
\end{example}

\begin{example}
    Another core example of a commutative idempotent semiring is the \textbf{boolean semiring}
    \[\mathbb B = (\{\bot, \top\}, \vee, \wedge, \bot, \top)\]
    where $\bot$ represents the logical false, $\top$ the logical true, $\vee$ is logical or and $\wedge$ is logical and.
\end{example}

\begin{remark}
    The boolean semiring is the unique idempotent semiring of order $2$.
\end{remark}

\begin{definition}\label{def:SemiringOrder}
    Idempotent semirings carry a natural partial order, where we define:
\[a \le b \iff a + b = a\]
    We say $a < b$ if $a \le b$ and $a \ne b$.
\end{definition}
This is in analogy with the tropical semiring, $\T = (\R \cup \{\infty\}, \min, +, \infty, 0)$. A useful trick to remember this convention is that ``the crocodile eats the larger element," as, if $b \ge a$, then $b$ disappears in the sum. Some authors prefer the flipped convention where $a \ge b$ if and only if $a + b = a$.

\begin{example}
    In $\mathbb B$ we have that $\top < \bot$.

    In $\T = (\R \cup \{\infty\}, \min, +_\R, \infty, 0)$, $a < b$ if $a$ is less than $b$ as elements of $\R$.

    If we instead use the semiring $(\R \cup \{-\infty\}, \max, +_\R, -\infty, 0)$, $a < b$ in our semiring if $a$ is greater than $b$ as elements of $\R$.
\end{example}

\begin{proposition}
    Semiring homomorphisms between idempotent semirings are order preserving.
\end{proposition}
\begin{proof}
    Let $\phi: S \to S'$ be a semiring homomorphism, and let $a \le b$, i.e. $a + b = a$.

    Because order is expressed by an algebraic equality, and semiring homomorphisms preserve algebraic equalities, we get:
    \[\phi(a) + \phi(b) = \phi(a + b) = \phi(a)\]
    i.e. $\phi(a) \le \phi(b)$.
\end{proof}

\begin{definition}
    Let $X \subseteq S$ be a finite subset of an idempotent semiring $S$. We define:
    \[\inf(X) = \sum_{x \in X} x\]

    In particular, for $a, b \in S$, we define
    \[\inf(a,b) = a + b\]
\end{definition}
\begin{remark}
    It is equivalent to define an idempotent semiring as a unital monoid $(S,*, 1_S)$ equipped with a lower lattice order and a minimal element $0_S$. The above definition gives this correspondence.
\end{remark}

\begin{definition}
    If the natural partial order on an idempotent semiring $S$ is a total order then we say that $S$ is a \textbf{totally ordered} idempotent semiring.
\end{definition}

Most idempotent semirings are not totally ordered. For instance, the power set of a monoid carries a non-totally ordered semiring structure:
\begin{example}\label{ex:powerset}
    Let $(M, *)$ be a monoid, not necessarily commutative, with identity $e$. We can construct the \textbf{power set} idempotent semiring of $M$ as the set of subsets of $M$, $2^M$, with addition being union and multiplication being Minkowski multiplication. That is:
    \[A + B = A \cup B\]
    \[A * B = \{ab : a \in A, b \in B\}\]
    The additive identity of $2^M$ is the empty set, $\emptyset$, and the multiplicative identity is the singleton, $\{e\}$.

    $A \le B$ if and only if $A \supseteq B$.
\end{example}

\begin{proposition}
    $2^M$ is totally ordered if and only if $M = \{e\}$.
\end{proposition}

\begin{definition}
    A \textbf{congruence} $C$ of a semiring $S$ is a subsemiring of $S \times S$ that is also an equivalence relation.
    \begin{itemize}
        \item For all $a \in S$, $(a,a) \in C$.
        \item If $(a,b), (c,d) \in C$ then $(a,d), (b,a), (a+c,b+d), (ac,bd) \in C$.
    \end{itemize}
\end{definition}

Congruences are closed under arbitrary intersection, which gives us the following definition:
\begin{definition}
    Given a set of relations $X \subseteq S \times S$, we say that the congruence generated by $X$, $\gen{X}$, is the intersection of all congruences which contain $X$.
\end{definition}

If our semiring $S$ is a ring then the set of congruences is in bijection with the set of ideals, as each congruence is determined by the equivalence class of $0$.

We note that requiring an equivalence relation $C$ to be a subsemiring of $S \times S$ causes the induced addition and multiplication maps on the equivalence classes to be well defined. This allows us to define a quotient of semirings.
\begin{definition}
    Given a semiring $S$ and a congruence $C$, we say the \textbf{quotient} $S/C$ is the semiring on the equivalence classes of $C$ equipped with the operations induced from $S$.
\end{definition}

A more detailed treatment of congruences and quotient semirings can be found in chapter 8 of \cite{golan1999semirings}.

It is important to note that for general semirings the quotients of a semiring $S$ are not determined by the equivalence class of $0$. As we do not have subtraction, there are many nontrivial quotients of a semiring such that the equivalence class of $0$ is just $\{0\}$.

\begin{definition}
    Given a semiring $S$, and a set of variables $X$, we say that the \textbf{semiring of expressions over $X$}, denoted $S\gen{X}$, and also called the \textbf{non-commutative polynomial semiring}, is the freely generated semiring over elements in $R$ and variables in $X$ with no relations between them apart from defining $1_S = 1_{S \gen{X}}$ and $0_S = 0_{S \gen{X}}$.

    So for instance we have elements in $\T\gen{x,y}$ of the form:
    \[3xy + 12x + x12 + 0\]
    Where $12x \ne x12$ and $xy \ne yx$.
\end{definition}.

\begin{definition}
    The \textbf{polynomial semiring} over variables $X$ is the quotient
    \[S\gen{X}/\sim\]
    Where $\sim$ is the congruence generated by the relations
    \begin{itemize}
        \item For $x, y \in X$, $xy \sim yx$.
        \item For $s \in S, x \in X$, $sx \sim xs$.
    \end{itemize}

    We denote this semiring $S[X]$.
\end{definition}

$S\gen{X}$ is the semiring of algebraic expressions where we do not assume that our variables commute with $S$ or with each other, whereas $S[X]$ is the semiring where we assume that our variables \textit{do} commute with $S$ and with each other.

\begin{proposition}
    Let $S$ be an idempotent semiring, $X$ any set, and $C$ any congruence on $S$. The following are all idempotent semirings:
    \begin{itemize}
        \item $S \gen {X}$
        \item $S[X]$
        \item $C$
        \item $S/C$
    \end{itemize}
\end{proposition}
This follows from proposition \ref{thm:idemp1}, as $1_S + 1_S = 1_S$ for each of these semirings.

A detailed reference on the general theory of semirings is the book \cite{golan1999semirings}.

\subsection{Valuations}
\begin{definition}\label{def:valuation}
Let $R$ be a ring and $\Gamma$ an idempotent semiring. We say that a function $\nu : R \to \Gamma$ is a \textbf{valuation} if $\nu$ is:
\begin{description}
    \item[Unital:] $\nu(0_R) = 0_\Gamma$, $\nu(1_R) = 1_\Gamma = \nu(-1_R)$,
    \item[Multiplicative:] $\nu(a*_Rb) = \nu(a)*_{\Gamma}\nu(b)$,
    \item[Superadditive:] $\nu(a +_R b) \ge \nu(a) +_{\Gamma} \nu(b) = \inf_{\Gamma}(\nu(a), \nu(b))$.
\end{description}
\end{definition}

From definition \ref{def:SemiringOrder} we get that superadditivity is equivalent to the following identity:
\[\nu(a +_R b) +_\Gamma \nu(a) +_{\Gamma} \nu(b) = \nu(a) +_{\Gamma} \nu(b) \]

The above definition aligns with the generalized semiring valuation in \cite{2016}, and from that paper we get another definition:
\begin{definition}
If $\nu^{-1}(0_\Gamma) = \{0_R\}$, we say that $\nu$ is \textbf{nondegenerate}.
\end{definition}

\begin{notation}
For the rest of this paper, we will use $\nu$ to refer to a valuation $R \to \Gamma$. Unless said otherwise, all valuations in this paper are nondegenerate.
\end{notation}

\begin{example}
Let $\Gamma = (\Z \cup \{\infty\}, \min, +_\Z, \infty, 0)$.

Fix a prime number $p$, we can form the $p$-adic valuation on $\Q$ via the following:

Each $x \in \Q$ can be written in the form $p^n * \frac{a}{b}$ for integers $a, b$ and $n$, where $p$ does not divide $a$ or $b$. We define
\[\nu_p(x) = n\]
$\nu_p: \Q \to \Gamma$ forms a valuation.

We can then construct a more general valuation on $\Q$ into $\Z^\omega = \{\infty\} \amalg  \{\oplus_{p \text{ prime}} \Z\}$ as the direct product of these valuations, where an element in $\Z^\omega$ is either $\infty$, or is a vector of finite support. Addition in $\Z^\omega$ between vectors is pointwise minimum, and multiplication is pointwise addition. We also state $a + \infty = a$ and $a * \infty = \infty$ for all $a \in \Z^\omega$.

We note that $\Z^\omega$ is isomorphic to the semiring $(\Q_{\ge 0}, \operatorname{gcd}, *)$: the nonnegative rational numbers where addition is taking the greatest common divisor of two rational numbers, and multiplication is standard multiplication. Here we mean the notion of greatest common divisor from number theory: for two rational numbers $x, y \in \Q$, we may write them over a common denominator, $x = \dfrac{a_x}{b}, y = \dfrac{a_y}{b}$ with $a_x,a_y,b \in \Z$. We say that $\operatorname{gcd}(x,y) = \operatorname{gcd}(a_x, a_y)$. One can see that our choice of common denominator does not affect the resulting greatest common divisor.

A vector $(a_p)_{p \text{ prime}} \in \Z^\omega$ is isomorphic to the element $\prod_{p \text{ prime}} p^{a_p} \in \Q_{\ge 0}$, with $\infty$ in $\Z^\omega$ being mapped to $0$.
\end{example}

The valuations above are intimately related to the description of Non-Archimedean absolute values on $\Q$, which we will discuss more in section \ref{sec:NonArchOstrow}. As a result, valuations are usually thought of in terms of absolute values, that is, measuring the size of various elements of our ring $R$. When we move away from $\Gamma$ being totally ordered, we can see how valuations characterize different sorts of underlying structures on our rings $R$.

\begin{example}\label{ex:solutions}
    Let $R$ be a ring without zero divisors and $A$ an $R$-algebra without zero divisors.

    Let $\Gamma = (2^{A^n}, \cap, \cup, A^n, \emptyset)$.

    We can form a valuation, $\nu$, on either the polynomial ring of $R$, $R[x_1,...,x_n]$, or the ring of expressions, $R\gen{x_1,...,x_n}$, with:
    \[\nu(f) = \{\vec{a} \in A^n: f(\vec{a}) = 0\}\]
\end{example}

\begin{example}
    Let $g, f$ be functions $\R \to \R$. We say that $g \in O(f)$ if there exists $k \ge 0, n \in \R$ such that $|g(x)| \le k|f(x)|$ for all $x \ge n$.

    Let $R$ be the ring of functions $\R \to \R$ equipped with pointwise addition and multiplication and let $2^R$ be the power set semiring on the multiplicative monoid of $R$.

    Let $\Gamma = (2^X/\sim, \cup, *, [\{0\}], [\{1\}])$, where we say that $\sim$ is the congruence generated by the relations: $X \sim Y$ if for all $f \in X$ there exists a $g \in Y$ such that $f \in O(g)$, and for all $g \in Y$ there exists an $f \in X$ such that $g \in O(f)$. We additionally state that $\{0\} \sim \emptyset$.

    The map
    \[f \mapsto [O(f)] = \{g : \exists k \ge 0, n \in \R : |g(x)| \le k|f(x)| \quad \forall x \ge n\}\]
    forms a valuation of $R$.
\end{example}

The following lemma also appears in \cite{2016} as Lemma 2.5.3 but we give another proof:
\begin{lemma}\label{thm:meetofsum}
Let $R$ be a ring, and let $\nu: R \to \Gamma$ be a valuation on $R$.

For any $a, b \in R$ the following three values are equal: 
\begin{itemize}
    \item $\nu(a) + \nu(b)$
    \item $\nu(a+b) + \nu(a)$
    \item $\nu(a+b) + \nu(b)$
\end{itemize}
\end{lemma}
\begin{proof}
First note that $\nu(a) + \nu(b) = \inf(\nu(a),\nu(b))$, i.e., for any $x$ such that $x \le \nu(a)$ and $x \le \nu(b)$ we have that $x \le \nu(a) + \nu(b)$.

By unitality and multiplicativity of $\nu$, we get $\nu(-x)=\nu(-1) \nu(x) = \nu(x)$.

We can write $a = a + b - b$, so superadditivity gives us that $\nu(a) \ge \nu(a+b) + \nu(-b) = \nu(a+b) + \nu(b)$. Similarly we get $\nu(b) \ge \nu(a + b) + \nu(a)$.

$\nu(a + b) + \nu(b)$ is less than both $\nu(a + b)$ and $\nu(a)$, so it must be less then their infinum, $\nu(a + b) + \nu(a)$.

$\nu(a + b) + \nu(a)$ is less than $\nu(a)$ and $\nu(b)$ so it must be less than $\nu(a) + \nu(b)$.

By superadditivity, $\nu(a) + \nu(b)$ is less than $\nu(a + b)$. It is also less than $\nu(b)$, so we get that $\nu(a) + \nu(b)$ is less than $\nu(a + b) + \nu(b)$.

This gives us a chain of inequalities which states that our three infinums are equal.

\[\nu(a) + \nu(b) \le \nu(a + b) + \nu(b) \le \nu(a + b) + \nu(a) \le \nu(a) + \nu(b)\]
\end{proof}

\begin{remark}
    The above theorem is not just a consequence of superadditivity, but is in fact equivalent to superadditivity. For a map $\nu : R \to \Gamma$ where $\Gamma$ is an idempotent semiring and $R$ any semiring, if for any $a, b \in R$ those three values are equal, then $\nu$ is superadditive.
\end{remark}

\begin{corollary}
    From the correspondence between the valuation on $\Q$, $\Z^\omega$, and the idempotent semiring $(\Q_{\ge 0}, \operatorname{gcd}, *)$, we get:
    \[\operatorname{gcd}(a,b) = \operatorname{gcd}(a - b, b)\]
\end{corollary}

\section{Universal Valuation Semiring}\label{sec:UnivSemi}
As noted in \cite{2016}, these valuations form a category, where a morphism between two valuations $\nu: R \to \Gamma$, $\nu':R' \to \Gamma'$ is a pair $(\phi, \rho)$ where $\phi : R \to  R', \rho : \Gamma \to \Gamma'$ are homomorphisms such that the following diagram commutes:
\[\begin{tikzcd}
	R & {R'} \\
	\Gamma & {\Gamma'}
	\arrow["\phi", from=1-1, to=1-2]
	\arrow["\rho", from=2-1, to=2-2]
	\arrow["\nu"', from=1-1, to=2-1]
	\arrow["{\nu'}", from=1-2, to=2-2]
\end{tikzcd}\]

If we take the subcategory of valuations from a fixed ring $R$, where each $\phi$ is the identity on $R$, \cite{2016} noted that there is an initial object. The construction in \cite{2016} assumed commutativity, but it can be easily generalized to a non-commutative case.

\begin{definition}\label{UnivVal}
    We construct the \textbf{universal valuation} semiring of $R$, $\Gamma_R$ as follows.

    For each $a \in R$, we define a variable $x_a$. We denote the free semiring of boolean expressions over these variables as
    \[\mathbb B\gen{x_R}\]

    We then define $\Gamma_R$\index{$\Gamma_R$} as
    \[\Gamma_R = \mathbb B\gen{x_R}/\sim\]
    Where $\sim$ is the congruence generated by the relations
    \begin{align*}
        x_{0} &\sim 0\\
        x_1 &\sim 1 \\
        x_{-1} &\sim 1\\
        x_{a}x_{b} &\sim x_{ab}\\
        x_{a + b} + x_a + x_b &\sim x_{a} + x_b
    \end{align*}

    The valuation associated with this semiring, is the map
    \[\nu(a) = [x_a]\]
\end{definition}

The only difference in this construction from the one that appears in \cite{2016} is that, rather than work over a quotient of the \textit{commutative} polynomial semiring $\mathbb B[x_R]$, we work over a quotient of the \textit{non-commutative} polynomial semiring $\mathbb B\gen{x_R}$. For a commutative ring $R$, we can see that $\Gamma_R$ will be commutative and identical to the Giansiracusas' construction.

\begin{proposition}
    $\Gamma_R$ equippled with the valuation $\nu(r) = x_r$ is initial over valuations of $R$. That is, for any valuation $\hat \nu : R \to \hat \Gamma$ there exists a unique semiring homomorphism $\phi: \Gamma_R \to \hat \Gamma$ such that $\hat \nu = \phi \circ \nu$.
\end{proposition}
\begin{proof}
    We construct $\phi$ by extending the map $\phi([x_r]) = \hat \nu(r)$ linearly.

    The fact that this is well defined follows from $\nu$ being a valuation.
\end{proof}

\begin{remark}
    If one is interested in what valuations \textit{could} exist for a given ring $R$, it suffices to study the hom-set of $\Gamma_R$, $\cat{ISR}(\Gamma_R, -)$.
\end{remark}

\begin{example}\label{ex:SmallestNC}
    Let $R$ be the ring of upper triangular $2 \times 2$ matrices over $\F_2$. $R$ has eight elements and they are generated by the matrices:
    \[i = \begin{bmatrix}1&0\\0&0\end{bmatrix} \qquad j = \begin{bmatrix}0&1\\0&0\end{bmatrix} \qquad k = \begin{bmatrix}0&0\\0&1\end{bmatrix}\]

    $\Gamma_R$ consists of $\mathbb B$ linear combinations of the elements: $0, 1, x_i, x_j, x_k, x_{i+j}, x_{j + k}, x_{i + j + k}$. Note that $x_{i + k} = x_1 = 1$.

    The multiplicative structure of $\Gamma_R$ is generated by the following multiplication table:
\begin{table}[H]
\centering
\begin{tabular}{|l|l|l|l|l|l|l|}
\hline
            & $x_i$ & $x_j$ & $x_k$     & $x_{i+j}$ & $x_{j+k}$ & $x_{i+j+k}$ \\ \hline
$x_i$       & $x_i$ & $x_j$ & $0$       & $x_{i+j}$ & $x_j$     & $x_{i+j}$   \\ \hline
$x_j$       & $0$   & $0$   & $x_j$     & $0$       & $x_j$     & $x_j$       \\ \hline
$x_k$       & $0$   & $0$   & $x_k$     & $0$       & $x_k$     & $x_k$       \\ \hline
$x_{i+j}$   & $x_i$ & $x_j$ & $x_j$     & $x_{i+j}$ & $0$       & $x_{i}$     \\ \hline
$x_{j+k}$   & $0$   & $0$   & $x_{j+k}$ & $0$       & $x_{j+k}$ & $x_{j+k}$   \\ \hline
$x_{i+j+k}$ & $x_i$ & $x_j$ & $x_{j+k}$ & $x_{i+j}$ & $x_k$     & $1$         \\ \hline
\end{tabular}
\end{table}

The additive structure has the following diagramatic form:
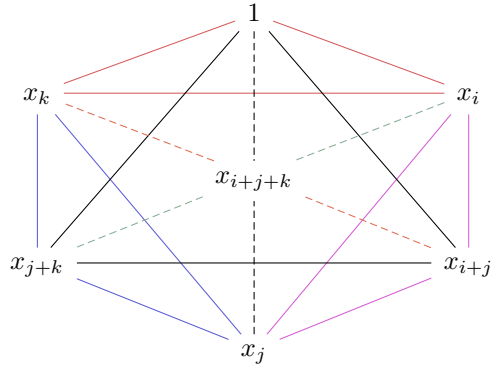
\[\begin{tikzcd}
	&& 1 \\
	{x_k} &&&& {x_i} \\
	&& {x_{i+j+k}} \\
	{x_{j+k}} &&&& {x_{i+j}} \\
	&& {x_j}
	\arrow[color={rgb,255:red,214;green,92;blue,92}, no head, from=2-1, to=2-5]
	\arrow[color={rgb,255:red,214;green,92;blue,214}, no head, from=2-5, to=5-3]
	\arrow[color={rgb,255:red,92;green,92;blue,214}, no head, from=2-1, to=5-3]
	\arrow[color={rgb,255:red,92;green,92;blue,214}, no head, from=2-1, to=4-1]
	\arrow[color={rgb,255:red,92;green,92;blue,214}, no head, from=5-3, to=4-1]
	\arrow[color={rgb,255:red,214;green,92;blue,214}, no head, from=2-5, to=4-5]
	\arrow[color={rgb,255:red,214;green,92;blue,214}, no head, from=5-3, to=4-5]
	\arrow[color={rgb,255:red,214;green,92;blue,92}, no head, from=1-3, to=2-1]
	\arrow[color={rgb,255:red,214;green,92;blue,92}, no head, from=1-3, to=2-5]
	\arrow[no head, from=4-5, to=1-3]
	\arrow[no head, from=4-1, to=1-3]
	\arrow[no head, from=4-1, to=4-5]
	\arrow[dashed, no head, from=1-3, to=3-3]
	\arrow[dashed, no head, from=5-3, to=3-3]
	\arrow[color={rgb,255:red,133;green,173;blue,150}, dashed, no head, from=2-5, to=3-3]
	\arrow[color={rgb,255:red,133;green,173;blue,150}, dashed, no head, from=3-3, to=4-1]
	\arrow[color={rgb,255:red,224;green,108;blue,82}, dashed, no head, from=3-3, to=4-5]
	\arrow[color={rgb,255:red,224;green,108;blue,82}, dashed, no head, from=3-3, to=2-1]
\end{tikzcd}\]

The sum of any two elements with the same arrow type and color are the same item. So, for instance, we have:
\[[x_i + x_k] = [x_k + 1] = [1 + x_i]\]
and
\[[1 + x_{i + j + k}] = [x_{i + j + k} + x_j] = [1 + x_j]\]

There is a single element corresponding to a non-degenerate sum of three elements, and it is obtained by the sum of any three elements such that the arrows between them are distinct colors or types, i.e.:
\[[x_i + x_j + x_k] = [1 + x_{i + j} + x_j]\]
But, for instance, the element:
\[[1 + x_{i + j} + x_{j + k}] = [1 + x_{i + j}] \ne [x_i + x_j + x_k]\]

An observant geometer may note that there is a striking similarity between the diagram above and the Fano plane. That is not a coincidence, as the additive structure of $\Gamma_R$ is exactly the matroid structure of the Fano plane. The higher ``cells" of our additive structure correspond to the additive $\Z$-linear subspaces of $R$ generated by multiple elements of $R$. We will see that this is not just a coincidence for this particular ring, but is a general fact about the structure of $\Gamma_R$ over any ring.
\end{example}

\subsection{Structure Theorem}\label{sec:structure}
\begin{theorem}[Structure Theorem for $\Gamma_R$]\label{thm:valladdstruct}
    Let $(a_i)_{i \in I}$ and $(b_j)_{j \in J}$ be finite collections of elements in a ring $R$. In $\Gamma_R$ we have:
    \[\left[\sum_{i \in I}x_{a_i}\right] = \left[\sum_{j \in J} x_{b_j}\right]\]
    if and only if $\Span_\Z(a_i) = \Span_\Z(b_j)$.
\end{theorem}

\begin{proof}
It suffices to prove that for any finite collection $(a_i)_{i \in I} \subseteq R$ and singleton $b \in R$
\[\left[\sum_{i \in I} x_{a_i}\right] \le [x_b] \iff b \in \Span_\Z(a_i)\]
If each $b_i$ is contained in $\Span_\Z(a_i)$ then $\Span_\Z(b_i) \subseteq \Span_\Z(a_i)$, and if $\left[\sum_{i \in I} x_{a_i}\right] \le [x_{b_i}]$ for all $i$, then $\left[\sum_{i \in I} x_{a_i}\right] \le \left[\sum_{i \in I} x_{b_i}\right]$. If the above statement hold for singletons, we get our original statement.

Let $b \in \Span_{\Z}(a_i)$, then we may write
\[b = \sum_{i \in I} n_i a_i\]
where $n_i \in \Z$ for each $i$.

By superadditivity and idempotence we get:
\[x_b \ge \sum_{i \in I} n_i x_{a_i} = \sum_{i \in I} x_{a_i}\]

Now will show that if if $b \not\in \Span_{\Z}(a_i)$ then $\left[\sum_{i \in I}x_{a_i}\right] \not\le x_b$.

We note that for any valuation $\nu : R \to \Gamma$ we can write $\nu$ as $\phi \circ \nu_*$ where $\nu_* : R \to \Gamma_R$ is our universal valuation and $\phi : \Gamma_R \to \Gamma$ is a semiring homomorphism.

As we noted in definition \ref{def:SemiringOrder}, semiring homomorphisms are order preserving. As valuations are the image of $\Gamma_R$ under a semiring homomorphism, if there is \textit{any} valuation such that $\nu(y) \not\le \nu(z)$ then $x_y \not\le x_z$.

So to show that $x_b \not\ge \left[\sum_{i \in I} x_{a_i}\right]$ it suffices to find a single valuation $\nu$ such that
\[\nu(b) \not\ge \sum_{i \in I} \nu(a_i)\]

To do so, we construct an idempotent semiring $\daleth$ where:
\[\daleth = 2^R /\sim\]
Here $2^R$ is the power set semiring of the multiplicative monoid of $R$ (example \ref{ex:powerset}).

We put an equivalence relation on $\daleth$ where $A \subseteq R$, $B \subseteq R$, $A \sim B$ if and only if $\Span_{\Z}(A) = \Span_{\Z}(B)$. We additionally say $\Span_\Z(\emptyset) = \{0\}$.

As $\sim$ is an  equivalence relation, to be a congruence we only need to show that if $A \sim B$ and $C \sim D$ then $A \cup C \sim B \cup D$ and $A * C \sim B * D$.

We note that, if $x \in \Span_\Z(A \cup C)$, then we may write $x = a + c$ for $a \in \Span_\Z(A)$ and $c \in \Span_\Z(C)$. As $A \sim B$ we get that $a \in \Span_\Z(B)$ and similarly as $C \sim D$ we get that $c \in \Span_\Z(D)$, so $a + c \in \Span_\Z(B \cup D)$.

If $x \in \Span_\Z(A*C)$, then we may write:
\[x = \sum_{a \in A, c \in C} n_{ac} ac\]
Where $n_{ac} \in \Z$ and only finitely many are nonzero. As $\Span_\Z(A) = \Span_\Z(B)$ and $\Span_\Z(C) = \Span_\Z(D)$, we may write each $a$ as:
\[a = \sum_{b \in B} n_{ab} b\]
and each $c$ as:
\[c = \sum_{d \in D} n_{cd} c\]
so we may write $x$ in the form:
\[x = \sum_{a \in A, c \in C, b \in B, d \in D} n_{ac}n_{ab}n_{cd}bd\]
As only finitely many of $n_{ac}, n_{ab}, n_{cd}$ are nonzero, we get that $x \in \Span_\Z(BD)$.

We can form a valuation on $R$ into $\daleth$:
\[\nu(a) = \left[\{a\}\right]\]
This is clearly multiplicative, and as $\Span_\Z(1) = \Span_\Z(-1)$ we get $\nu(1) = \nu(-1) = \left[\{1\}\right]$, so we can see that $\nu$ is unital.

$\nu$ is superadditive, $\nu(a + b) + \nu(a) + \nu(b) = \nu(a) + \nu(b)$, as $\Span_\Z(a, b, a + b) = \Span_\Z(a, b)$.

Thus $\nu$ is a valuation of $R$, however, we see that $\nu(b) \not\ge \sum_{i \in I} \nu(a_i)$, as $b \not\in \Span_\Z(a_i)$.

Thus if $b \not\in \Span_\Z(a_i)$, then $x_b \not\ge \left[\sum_{i \in I} x_{a_i}\right]$




Thus we have a valuation where $\nu(b) \not\ge \sum_{i \in I}\nu(a_i)$, and so our theorem holds.
\end{proof}

This theorem and the construction of $\daleth$ give another characterization of $\Gamma_R$: $\Gamma_R$ is isomorphic to the subsemiring of $\daleth$ given by finite sets.

\begin{corollary}
$\Gamma_{\Q} \cong (\Q_{\ge 0}, \operatorname{gcd}, *, 0, 1) \cong (\Z^\omega \cup \{\infty\}, \min, +, \infty, \vec{0})$
\end{corollary}
\begin{proof}
    We note that for $a, b \in \Q_{\ge 0}$ we have that both $a, b \in \Span_{\Z}(\operatorname{gcd}(a,b))$ and $\operatorname{gcd}(a,b) \in \Span_\Z(a,b)$, so the theorem above gives us that in $\Gamma_\Q$:
    \[\left[x_a + x_b\right] = \left[x_{\operatorname{gcd}(a,b)}\right]\]

    We also get that $[x_a] = [x_b]$ if and only if $a = \pm b$, and so each equivalence class in $\Gamma_\Q$ corresponds uniquely to an element in $\Q_{\ge 0}$, and this correspondence is a semiring isomorphism.

    As stated above, $(\Z^\omega \cup \{\infty\}, \min, +, \infty, \vec{0}) \cong (\Q_{\ge 0}, \operatorname{gcd}, *, 0, 1)$
\end{proof}

\begin{corollary}
    For all prime $p$, $\Gamma_{\F_p} \cong \mathbb B$
\end{corollary}
\begin{proof}
    We note that for any $i, j \in \F_p \setminus \{0\}$, there exists an $m \in \N$ such that $mi \equiv j \mod p$, so $[x_i] \le [x_j]$.

    Thus for all $i \ne 0$, $[x_i] = [1]$.
\end{proof}

\begin{corollary}
    $\Gamma_{\F_{p^k}} \cong \mathbb B[x]/\gen{x^k = 1}$.
\end{corollary}
\begin{remark}
    $\Gamma_{\F_{p^k}}$ is also isomorphic to $\mathbb B^k$. with addition being pointwise or, and multiplication being defined by linearly extending the circular shift operators: If the vector $1_a$ is the vector with $1$ only in the $a$ coordinate, then $1_a * 1_b = 1_{a + b \mod k}$, where our coordinates are in the set $\{0,1,2,...,k-1\}$.

    $\Gamma_{\F_{p^k}}$ is also isomorphic to the powerset semiring $2^{\Z_k}$, where $\Z_k$ is the cyclic group with $k$ elements, our addition is union, and our multiplication is the Minkowski sum of our sets.
\end{remark}

\begin{proposition}
    $R \mapsto \Gamma_R$ defines a functor from the category of rings to the category of idempotent semirings.
\end{proposition}
\begin{proof}
    We have our definition of our functor on our objects. It remains to define our action on morphisms.

    Let $f : R \to R'$ be a map of rings, we define a map $\Gamma_f : \Gamma_R \to \Gamma_{R'}$ by linearly extending the maps $\Gamma_f(x_a) = x_{f(a)}$.

\[\begin{tikzcd}
	R && {R'} \\
	{\Gamma_R} && {\Gamma_{R'}}
	\arrow[""{name=0, anchor=center, inner sep=0}, "f", from=1-1, to=1-3]
	\arrow[""{name=1, anchor=center, inner sep=0}, "{\Gamma_f}"', from=2-1, to=2-3]
	\arrow[shorten <=4pt, shorten >=4pt, maps to, from=0, to=1]
\end{tikzcd}\]

    Theorem \ref{thm:valladdstruct} gives that $\Gamma_{f}$ is well defined, as if $\Span_\Z(a_i) = \Span_\Z(b_i)$ then $\Span_\Z(f(a_i)) = \Span_\Z(f(b_i))$, and so $[\sum x_{f(a_i)}] = [\sum x_{f(b_i)}]$ if $[\sum x_{a_i}] = [\sum x_{b_i}]$.

    By construction our functor respects composition and identity.
\end{proof}


\subsection{Ostrowski's Theorem and Non-Archimedean Absolute Values}\label{sec:NonArchOstrow}
The structure theorem when applied to $\Gamma_\Q$ gives us the Non-Archimedean case of Ostrowski's theorem on the equivalence classes of the absolute values of $\Q$.

\begin{definition}
    An absolute value on $\Q$ is a function:
    \[|\cdot| : \Q \to \R\]
    such that the following hold:
    \begin{description}
        \item[Non-Negative:] $|x| \ge 0$,
        \item[Positive Definite:] $|x| = 0$ if and only if $x = 0$,
        \item[Multiplicative:]$|xy| = |x||y|$,
        \item[Triangle Inequality:]$|x + y| \le |x| + |y|$.
    \end{description}

    An absolute value is \textbf{Non-Archimedean} if $|x + y| \le \max\{|x|,|y|\}$.

    We say that two absolute values $|\cdot|$ and $|\cdot|_*$ are equivalent if there exists a $c \in \R$, $c > 0$, such that:
    \[|x|_* = |x|^c\]
    for all $x \in \Q$.
\end{definition}

\begin{corollary}[\cite{OstrowskiTheorem}]\label{thm:Ostrow}
Up to equivalence, the only non-trivial absolute values on $\Q$ are either the standard absolute value or the $p$-adic ones.
\end{corollary}

The standard proof of this theorem involves quite a bit of symbol pushing. We will show how using theorem \ref{thm:valladdstruct} we can create another argument that covers the Non-Archimedean absolute value case.

\begin{proof}
    It can be shown that every Archimedean absolute value is equivalent to the standard one. We will focus on the Non-Archimedean case to show the power of our structural theorem.

    We note that the Non-Archimedean absolute values on $\Q$ correspond to non-degenerate valuations $\Q \to \T$.

    Given an absolute value $|\cdot|$ we can form a valuation:
    \[x \mapsto -\log|x|\]

    Given a non-degenerate valuation $\nu : \Q \to \T$ we can form an absolute value:
    \[|x| = e^{-\nu(x)}\]

    Let $\eta_c : \T \to \T$, for $c > 0$, be the tropical automorphism sending $x$ to the ``tropical $x^c$" or in standard arithmetic: $c \cdot x$.
    
    We note that two absolute values are equivalent if and only if there exists a $c > 0$ such that for their corresponding valuations we have $\nu = \eta_c \circ \nu'$.

    So then the question of what Non-Archimedean absolute values occur can be answered by examining the maps $\cat{ISR}(\Gamma_\Q, \T)$ where $\phi^{-1}(\infty) = 0$.

    Theorem \ref{thm:valladdstruct} tells us that $\Gamma_\Q \cong \Z^\omega$. We can see that for each prime $p$, the projection onto the $p$-th coordinate in $\Z^{\omega}$ gives us a non-degenerate map $\Z^{\omega} \to \T$ which is not equivalent to any other projection. These projections correspond to the $p$-adic absolute value on $\Q$.

    All that remains to show is that any non-degenerate map of semirings is equivalent to either one of these projections or the trivial map:
    \[x \mapsto \begin{cases}\infty & x = \infty\\ 0 & x \ne \infty\end{cases}\]

    Let $\phi : \Z^\omega \to \T$ be such that there exists $v \ne \infty$ such that $\phi(v) \ne 0$, i.e. $\phi$ corresponds to a non-trivial absolute value.

    We can write $v = \prod_{p} 1_p^{v_p}$ where $1_p$ is the vector that is $1$ only at the index corresponding to prime $p$. We then get:
    \[\phi(v) = \prod_{p} \phi(1_p)^{v_p}\]
    Which can be written using classical arithmetic as:
    \[\phi(v) = \sum_{p} v_p \cdot \phi(1_p)\]

    As $\phi(v) \ne 0$ we get that there must be some $\phi(1_p) \ne 0$.

    We note that if $\phi$ is such that there exists exactly one prime $p$ such that $\phi(1_p) \ne 0$ then $\phi$ is equivalent to the projection map $\pi_p$, as $\pi_p = \eta_{\frac{1}{\phi(1_p)}} \circ \phi$.

    There cannot be two primes $p$ and $q$ such that $\phi(1_p) \ne 0$ and $\phi(1_q) \ne 0$. $1_p + 1_q = 0$, so $\phi(1_p + 1_q) = \phi(0) = 0$ which cannot happen as if $\phi(1_p) \ne 0$ and $\phi(1_q) \ne 0$ then
    \[\phi(1_p + 1_q) = \phi(1_p) + \phi(1_q) = \min(\phi(1_p), \phi(1_q)) \ne 0\]

    Thus every Non-Archimedean absolute value is either trivial or equivalent to a $p$-adic absolute value.
\end{proof}

We note that this idea can be generalized, and that the Non-Archimedean absolute values on a ring $R$ are determined by the non-degenerate semiring homomorphisms $\Gamma_R \to \T$, with equivalences of absolute values correspond to the equivalence classes of those maps under the action given by postcomposition by the automorphisms $\eta_c : \T \to \T$ for $c > 0$.

A map $\Gamma_R \to \T$ has commutative image, and so it factors through the abelianization of $\Gamma_R$: $\Gamma_R/\gen{ab \sim ba}_{a,b \in \Gamma_R}$. The requirement that $\Gamma_R$ is a nondegenerate map imposes a requirement that $R$ is commutative. So, if one were to look at our ring $R$ from example \ref{ex:SmallestNC}, we can see that there are no maps of semirings $\Gamma_R \to \T$, and so there can be no Non-Archimedean absolute values on $R$. Here we can get an even stricter statement: there are not even any Non-Archimedean \textit{seminorms} on $R$, as the abelianization of $R$ is congruent to the zero ring and there are no maps at all between $\Gamma_0$ and $\T$.

\begin{definition}
    Let $X$ be a ring or semiring, we define the \textbf{abelianization} of $X$, $\operatorname{Ab}(X)$, to be the quotient of $X$ by the relation generated by:
    \[ab \sim ba\]
    for $a, b \in X$.
\end{definition}

As $\operatorname{Ab}(\mathbb B\gen{X}) \cong \mathbb B[X]$, we can explicitly relate the construction $\Gamma_R$ to the one given in \cite{2016}. The construction in \cite{2016} is exactly $\operatorname{Ab}(\Gamma_R)$. Furthermore we have the following isomorphism

\begin{proposition}
    \[\operatorname{Ab}(\Gamma_R) \cong \Gamma_{\operatorname{Ab}(R)}\]
\end{proposition}

As Non-Archimedean norms are in correspondence with valuations $R \to \T$, which factor through $\Gamma_{\operatorname{Ab}(R)}$, this isomorphism tells us that Non-Archimedean absolute values cannot exist on non-commutative rings, as they cannot be positive definite. In addition it tells us that Non-Archimedean seminorms on a ring $R$ are in bijection with the Non-Archimedean seminorms on $\operatorname{Ab}(R)$.

\begin{corollary}
    Let $R$ be a ring and $|\cdot| : R \to \R$ a map such that for all $x, y \in R$
    \begin{description}
        \item[Nonnegative] $|x| \ge 0$
        \item[Multiplicative] $|xy| = |x||y|$
        \item[Non-Archimedean] $|x + y| \le \max(|x|, |y|)$
    \end{description}

    If $R$ is non-commutative then $|\cdot|$ is not positive definite.
\end{corollary}
\begin{proof}
    A nonnegative, multiplicative, Non-Archimedean map forms a valuation into the semiring $\gimel = ([0,\infty], \max, *, \infty, 1)$, thus $|\cdot|$ corresponds to some idempotent semiring homomorphism $\phi : \Gamma_R \to \gimel$. We note that as $\gimel$ is commutative, this map factors through $\operatorname{Ab}(\Gamma_R) \cong \Gamma_{\operatorname{Ab}(R)}$.

    Let $\phi = \tilde \phi \circ \pi$ where $\pi : \Gamma_R \to \Gamma_{\operatorname{Ab}(R)}$ is the induced map from the quotient $R \to \operatorname{Ab}(R)$ and $\tilde \phi : \Gamma_{\operatorname{Ab}(R)} \to \gimel$.

    Choose an $y, z \in R$ such that $yz \ne zy$. $yz - zy \ne 0 \in R$, but $[yz - zy] = [0] \in \operatorname{Ab}(R)$. $\pi(x_{yz - zy}) = 0$ and so $\rho \circ \pi(x_{yz - zy}) = 0$.

    Thus $\phi(x_{yz - zy}) = 0$ so $|yz - zy| = 0$ and $|\cdot|$ is not positive definite.
\end{proof}

From this we can see that a non-commutative ring has no Non-Archimedean absolute values, and all Non-Archimedean seminorms on $R$ correspond to Non-Archimedean seminorms on $\operatorname{Ab}(R)$

\section{Non-Commutative Valuations}
The core difference in the construction in definition \ref{UnivVal} versus that in \cite{2016} is that it applies to non-commutative rings and allows for non-commutative valuations. One may wonder then if there are any useful examples of non-commutative valuations.

\subsection{Examples of Non-Commutative Valuations} \label{sec:NoncommVal}

\begin{example}\label{ex:centralizers}
    Let $R$ be a non-commutative division ring, and let $Z(R)$ be the center of $R$. The elements which ``centralize" elements of $R$ form a valuation on $R$.

    Let $\Gamma$ be composed of sets such that, for $X \in \Gamma$, $\{0\} \subseteq X \subseteq R$, and for all $x \in X$ the coset $x * Z(R)$ is contained in $X$.

    We define our operations and identities of $\Gamma$ as

    \begin{align*}
        X + Y &= X \cap Y\\
        X * Y &= \{y * x : y \in Y, x \in X\}\\
        0_\Gamma &= R\\
        1_\Gamma &= Z(R)
    \end{align*}
    Note that our multiplication is not the standard Minkowski multiplication.

    Consider the map $\nu$ with
    \[\nu(x) = \{y \in R : xy \in Z(R)\}\]

    $\nu$ forms a valuation $R \to \Gamma$.

    The requirement that $R$ is a division ring is necessary for $\nu$ to be multiplicative. Subadditivity is given by the distributive property on $R$, and the unitality of $\nu$ is a result of $Z(R)$ being a subring of $R$.

    This construction is not specific to $Z(R)$, if $K$ is \textit{any} commutative subring of a division ring we can replace $Z(R)$ with $K$ in the above construction and obtain a valuation on $R$.
\end{example}

Traditionally valuations are though of as measures of size of elements of a ring $R$. We can see through these example as well as in example \ref{ex:solutions} that valuations also encode information on solutions of equations. While imagining a non-commutative notion of size is nontrivial, there are many non-commutative equations and relations that are useful to study.

\begin{example}
    Let $R$ be a ring, and let $2^R$ be the power set semiring of the multiplicative monoid of $R$.

    Consider the equivalence relation: $A \sim B$ if and only if the two sided ideal generated by $A$ is the same as the two sided ideal generated by $B$. This forms a congruence, as it is closed under the operations of $2^R$.

    Let $\Gamma = 2^R/\sim$. The map
    \[a \mapsto [\{a\}]\]
    forms a valuation on $R$, and we note that if $R$ is non-commutative then $\Gamma$ will be as well.
\end{example}

\subsection{Non-Commutative Tropicalization and the Roots of Expressions}\label{sec:tropicalizations}

One of the key uses of valuations classically is that they allow for \textit{tropicalization}. In classical tropical geometry there is a correspondence between solutions of polynomial expressions in a valuated ring, and ``crease points" of polynomials over the tropical semiring.

With a notion of non-commutative valuations, we can extend this correspondence, and use non-commutative idempotent semirings to study the solutions of non-commutative equations.

\begin{definition}[$\nu$-Tropicalization]\index{Tropicalization}
    Let $R$ be a ring, $\Gamma$ an idempotent semiring, and $\nu : R \to \Gamma$ a valuation.
    
    For an expression $f \in R\gen{X}$ we define its \textbf{$\nu$-Tropicalization} as the expression $\trop_\nu(f) \in \Gamma\gen{X}$ obtained by replacing every coefficient by its valuation.
\end{definition}
\begin{example}
    Consider the expression $x^3 * 12 * x - 2 * x + x * 2 \in \Q\gen{x}$ with the $2$-adic valuation. The tropicalization of this expression is:
    \[x^3 * 2 * x + 1*x + x*1 \in \T\gen{x}\]
\end{example}

\begin{definition}
    Given an expression in an idempotent semiring $f \in \Gamma\gen{x_1, ..., x_n}$. If we write $f$ as the sum of monomials:
    \[f = \sum_{i} m_i\]

    We say that a \textbf{crease point} of $f$ is a vector $\vec{z} \in \Gamma^n$ such that for all $k$:
    \[f(\vec{z}) = \sum_{i \ne k} m_i(\vec{z})\]

    This is referred to as the \textit{bend relation} in \cite{2016}, and is a generalization of the definition of a tropical hypersurface which may be read about in \cite{MaclaganSturmfels}.

    If $\Gamma$ is totally ordered, this can be rephrased to the classical definition in tropical geometry: $\vec{z}$ is a crease point if and only if the minimum
    \[f(\vec{z}) = \min_i(m_i(\vec{z}))\]
    is attained twice.
\end{definition}

\begin{theorem}
    Let $\vec{x}$ be such that $f(\vec{x}) = 0$, then $\nu(\vec{x})$, the pointwise valuation of $\vec{x}$ is a crease point of $\trop_\nu(f)$.
\end{theorem}
\begin{proof}
    \[\trop_\nu(f)(\nu(\vec{x})) = \sum \nu(m_i(x))\]

    By inducting on lemma \ref{thm:meetofsum}, we get:
    \[\sum_{i = 1}^n \nu(a_i) = \nu\left(\sum_{i = 1}^n a_i\right) + \sum_{j \ne k} \nu(a_j)\]
    But as $\sum_{i = 1}^n a_i = 0$ we get $\nu\left(\sum_{i = 1}^n a_i\right) = 0_\Gamma$ and so
    \[\nu\left(\sum_{i = 1}^n a_i\right) + \sum_{j \ne k} \nu(a_j) = \sum_{j \ne k} \nu(a_j)\]

    As $k$ was arbitrary, we can remove any monomial of $\trop_\nu(f)(\nu(\vec{x}))$ and so $\nu(\vec{x})$ is a crease point of $\trop_\nu(f)$.
\end{proof}

\begin{example}
    Let $R$ be our non-commutative ring from example \ref{ex:SmallestNC}, and consider the universal valuation $\Gamma_R$ on it.

    Consider the non-commutative expression
    \[f(z) = (j+k)z^2 + zk + j\]
    The roots of this expression are $1, j, k$ and $i + j$.

    \[ \trop_\nu(f)(z) = x_{j+k} z^2 + z x_k + x_j\]
    We get the following when we plug in the entries of $\Gamma_R$ corresponding to valuations of elements of $R$:
    \begin{align*}
        \trop_\nu(f)&(0) = 0 + 0 + x_j\\
        \trop_\nu(f)&(1) = x_{j+k} + x_k + x_j\\
        \trop_\nu(f)&(x_i) = x_i + 0 + x_j\\
        \trop_\nu(f)&(x_j) = 0 + x_j + x_j\\
        \trop_\nu(f)&(x_k) = x_{j+k} + x_k + x_j\\
        \trop_\nu(f)&(x_{i+j}) = 0 + x_j + x_j\\
        \trop_\nu(f)&(x_{j + k}) = x_{j+k} + x_{j+k} + x_j\\
        \trop_\nu(f)&(x_{i+j+k}) = x_{j+k} + x_{j+k} + x_j
    \end{align*}

    We can see that $x_k$ is a crease point, as if we delete our monomials we get $x_{j+k} + x_k$ or $x_{k} + x_j$ or $x_{j+k} + x_j$, which by theorem \ref{thm:meetofsum} are all identical. $x_{i + j + k}$ however is not, as if we delete the third monomial of $\trop_\nu(f)$ we get $x_{j + k}$, which is not the same as $x_{j + k} + x_j$.

    Here we can see that we have crease points $1, x_j, x_k, x_{i+j}$ which align with the roots of our expression.

    There are additional crease points in the higher sums of $\Gamma_R$, for instance $[x_i + x_j + x_k]$ is a crease point of $\trop_\nu(f)$. We suspect that this is due to behavior in an algebraic extension of $R$ as in the case with tropical geometry.
\end{example}

\subsection{Further Generalizations and Applications to Representations of Ultrametric Spaces}\label{sec:representations}

\begin{notation}
    We denote by $M_n(S)$ the semiring of $n \times n$ matrices over $S$, with addition being pointwise addition and multiplication being the standard multiplication of matrices.
\end{notation}

A well known classification of finite ultrametric spaces \cite{Leclerc1981} can be restated as: Every finite ultrametric space's distance metric can be represented as the least weights of paths in a graph with weights in the semiring $\Gamma = ([0,\infty], \min, \max, \infty, 0)$. This states that the finite ultrametric spaces are in correspondence with the solutions to matrix equations of the form
\[X = AX + I\]
Where $A \in M_n(\Gamma)$

This hints at a connection between ultrametric spaces and the geometry of idempotent semirings through matrices over idempotent semirings. We can further realize this connection by generalizing our notion of a valuation even further and see that the representations of a ring $R$ in an ultrametric vector space can be seen by examining the maps from idempotent semiring built from $R$, similar to $\Gamma_R$, to $M_n(\T)$.

We note that the two features used in the proof of theorem \ref{thm:valladdstruct} were that $x_{a + b} \ge x_a + x_b$ and the fact that $\Gamma_R$ was initial. We did not rely on the other axioms of a valuation, so if we loosen the axioms of a valuation even further we can construct similar initial objects and theorem \ref{thm:valladdstruct} will still apply. For instance, we could relax our multiplicativity condition to be a \textit{supermultiplicativity} condition instead, where
\[\nu(ab) \ge \nu(a) \nu(b)\]
rather than
\[\nu(ab) = \nu(a)\nu(b)\]

\begin{definition}
    Let $R$ be a ring and $\Gamma$ an idempotent semiring. We say a map $\nu : R \to \Gamma$ is a \textbf{super-multiplicative valuation} if $\nu$ is:
    \begin{description}
        \item[Unital:] $\nu(0_R) = 0_\Gamma$, $\nu(1_R) = 1_\Gamma = \nu(-1_R)$
        \item[Supermultiplicative:] $\nu(a*_Rb) \ge \nu(a)*_{\Gamma}\nu(b)$
        \item[Superadditive:] $\nu(a +_R b) \ge \nu(a) +_{\Gamma} \nu(b) = \inf_{\Gamma}(\nu(a), \nu(b))$
    \end{description}
\end{definition}

\begin{definition}
    We can construct a universal super-multiplicative valuation semiring, $\widehat \Gamma_R$, as the quotient of $\mathbb B\gen{x_R}$ by the congruence generated by the relations:
    \begin{align*}
        x_{0} &\sim 0\\
        x_1 &\sim 1 \\
        x_{-1} &\sim 1\\
        x_{ab} + x_{a}x_{b} &\sim x_{a}x_{b}\\
        x_{a + b} + x_a + x_b &\sim x_{a} + x_b
    \end{align*}

    $\widehat \Gamma_R$ has a canonical super-multiplicative valuation:
    \[\nu(r) = x_r\]
    
    Any super-multiplicative valuation $\rho : R \to \Gamma$ factors through a unique semiring homomorphism $\phi: \widehat \Gamma_R \to \Gamma$, i.e. $\rho = \phi \circ \nu$
\end{definition}
\begin{proposition}
        Let $(a_i)_{i \in I}$ and $(b_j)_{j \in J}$ be finite collections of elements in a ring $R$. In $\widehat\Gamma_R$ we have:
    \[\left[\sum_{i \in I}x_{a_i}\right] = \left[\sum_{j \in J} x_{b_j}\right]\]
    If and only if $\Span_\Z(a_i) = \Span_\Z(b_j)$
\end{proposition}
As each valuation is a super-multiplicative valuation, the proof for theorem \ref{thm:valladdstruct} applies here as well. Without multiplicativity we cannot always reduce to the case of examining singletons, and so this does not give a complete classification of the additive structure of $\widehat \Gamma_R$ unlike in the case of theorem \ref{thm:valladdstruct}.

Just as we could use $\cat{ISR}(\Gamma_R, \T)$ to detect non-archimedean absolute values on $R$, we can use $\cat{ISR}(\widehat\Gamma_R,M_n(\T))$ to detect representations of $R$ in ultrametric vector spaces.

\begin{example}
    Let $V$ be a $\Bbbk$ vector space equipped with an ultrametric norm and a choice of basis $\mathcal B = \{v_1, v_2, ...\}$, and $R$ a noncommutative ring with a representation in $V$:
    \[\varphi : R \to \operatorname{End}(V)\]

    We can create a super-multiplicative valuation $\nu : R \to M_{|\mathcal B|}(\T)$.

    For $r \in R$, write $\varphi(r)(v_i) = \sum_{v_j \in \mathcal B} \alpha^i_j v_j$. We then define $\nu(r)$ to be the matrix where
    $\nu(r)_{ij} = -\log |\alpha^i_j|$.

    We note that this is super-multiplicative valuation as
    \[\nu(r_1r_2)_{ij} = -\log|\sum \alpha^k_j\alpha^i_k|\]
    while
    \[(\nu(r_1)\nu(r_2))_{ij} = \min(-\log|\alpha^k_j|-\log|\alpha^j_k|) = -\max(\log(|\alpha^k_j \alpha^j_k|))\]

    We note that as our norm is an ultrametric norm this gives us super-multiplicativity.

    If $\varphi$ is a nontrivial representation then $\nu$ is a nontrivial super-multiplicative valuation.
\end{example}

The above construction tells us that each $n$ dimensional ultrametric representation of a ring $R$ induces a valuation $\nu: R \to M_n(\T)$.
\begin{proposition}
    If there are no nontrivial semiring homomorphisms $\widehat \Gamma_R \to M_n(\T)$ then there are no nontrivial $n$ dimensional ultrametric representations of $R$.
\end{proposition}

Here it is important that we allow for $\widehat \Gamma_R$ to be noncommutative, as $M_n(\T)$ is noncommutative.



\printbibliography

\end{document}